\documentclass[11pt,reqno]{amsart}
 
\usepackage{setspace,amscd,amsfonts,amssymb,amsmath,amsthm,hyperref}
\usepackage[all,cmtip]{xy}

\usepackage{geometry}
\usepackage{etoolbox}
\allowdisplaybreaks

\theoremstyle{plain}
\newtheorem{thm}{Theorem}[subsection]

\newtheorem{lem}[thm]{Lemma}
\newtheorem{prop}[thm]{Proposition}
\theoremstyle{definition}
\newtheorem{defn}[thm]{Definition}
\newtheorem{rem}[thm]{Remark}

\numberwithin{equation}{subsection}

\begin{document}
  
\title[Weyl  Modules  and  Weyl  Functors  for  hyper--map algebras]{Weyl  Modules  and  Weyl  Functors \\  for  hyper--map algebras}  
\author{Angelo Bianchi} 
\address{Universidade Federal de S\~ao Paulo - UNIFESP - Department of Science and Technology, Brazil}
\email{acbianchi@unifesp.br}
\thanks{A.B. is partially supported by CNPq (462315/2014-2) and FAPESP (15/22040-0 and 18/23690-6).}

\author{Samuel Chamberlin}
\address{Department of Mathematics and Statistics\\
Park University\\
Parkville, MO 64152}
\email{samuel.chamberlin@park.edu}

\begin{abstract}
We investigate the representations of the hyperalgebras associated to the map algebras $\mathfrak g\otimes \mathcal A$, where $\mathfrak g$ is any finite-dimensional complex simple Lie algebra and $\mathcal A$ is any associative commutative unitary algebra with a multiplicatively closed basis. We consider the natural definition of the local and global Weyl modules, and the Weyl functor for these algebras. Under certain conditions, we prove that these modules satisfy certain universal properties, and we also give conditions for the local or global Weyl modules to be finite-dimensional or finitely generated, respectively.

\end{abstract}

\maketitle

\section*{Introduction}

Let $\mathfrak g$ be a finite-dimensional Lie algebra and $\mathcal X$ be an affine scheme (for instance, an algebraic variety), both defined over a field $\mathbb K$, often assumed to be algebraically closed and with characteristic zero. The Map algebra is the Lie algebra of regular maps from $\mathcal X$ to $\mathfrak g$. Denoting by $\mathcal A$ the coordinate ring of $\mathcal X$, the map algebra can also be realized as the Lie algebra $\mathfrak g\otimes \mathcal A$. These algebras generalize the loop and current algebras, which play an important role in the theory of affine Kac-Moody Lie algebras. The representation theory of the map algebras is an extremely active area of research. Recently, there has been an intensive study of the finite-dimensional representation theory of the map algebras $\mathfrak g\otimes \mathcal A$, where $\mathfrak g$ is a finite-dimensional simple Lie algebra over the complex numbers and $\mathcal A$ is an associative commutative unitary algebra.  

Parallel to this, the hyperalgebras in positive characteristic are constructed by considering an integral form of the universal enveloping algebra of a Lie algebra and then tensoring this form over $\mathbb Z$ with an arbitrary field $\mathbb F$, which we must assume is algebraically closed. In the case of a complex simple Lie algebra $\mathfrak g$, Kostant \cite{K2} constructed such an integral form of $U(\mathfrak g)$. The corresponding hyperalgebra is usually denoted by $U_\mathbb F(\mathfrak g)$. The affine analogues of Kostant's form were constructed by Garland \cite{G} in the non-twisted affine Kac-Moody algebra case and by Mitzman \cite{M} in a more general way for all affine Kac-Moody algebras. Suitable integral forms for the map algebras were formulated by the second author in \cite{C}.

The global Weyl Modules introduced via generators and relations in the context of affine Lie algebras in \cite{CP} are parameterized by a dominant integral weight $\lambda$ of a complex semisimple Lie algebra $\mathfrak g$, denoted by $W(\lambda)$. They are infinite--dimensional if $\lambda\ne 0$ and $W(\lambda)$ is also a right module for a polynomial algebra $A_\lambda$ specifically constructed. Once the global Weyl modules were defined, the local Weyl modules were obtained by tensoring the global Weyl modules with irreducible modules for $A_\lambda$ or, equivalently, they can be given via generators and relations. After that, in \cite{FL} a more general case was considered where $\mathcal A$ was the coordinate ring of an algebraic variety and partial results analogous to those in \cite{CP} were obtained. Finally, in \cite{CFK} (motivated by \cite{CG}), the authors took a categorical approach to these modules for the map algebras and it was shown that there is a natural definition of the local and global Weyl modules via homological properties. This leads to the so called Weyl functor from the category of left modules of a commutative algebra to the category of modules for a Lie algebra, similarly to \cite{CP}. Such a functor was not considered in \cite{FL}. An advantage of this approach (considering the commutative algebra  $A_\lambda$) is the study of the structure of the global Weyl modules $W(\lambda)$, which are free $A_\lambda$-modules (see \cite{CP} for $\mathfrak{sl}_2$, \cite{CL} for $\mathfrak{sl}_{r+1}$, \cite{FoL} for simply laced algebras, and the general case by passing to the quantum setting using \cite{Kash} and \cite{BN}). 

Furthermore, in a zero-characteristic base field, the Weyl modules were widely considered in several contexts: for twisted affine Lie algebras in \cite{FMS,CIK}, for the general context of equivariant map algebras in \cite{FMS2,FKKS}, for the Lie superalgebras in \cite{BCM,CLS,FM} as well as some considerations involving Levi subalgebras in \cite{F3}, and for the hyperspecial current algebras in \cite{CIK}. There is an application to invariant theory in \cite{CL} and a few others in \cite{CKO,N}.

The corresponding hyperalgebras in the (twisted and un-twisted) affine case and their finite-dimensional representations, with an emphasis on the local Weyl modules, were studied by Moura and Jakelic in \cite{JM} and by Moura and the first author in \cite{BM}. 

In these papers the authors considered the positive characteristic analogues of local Weyl modules and explored the universal properties of these modules. Recently, the multicurrent and multiloop cases were considered by the authors in \cite{BiC}. However, there is no mention of global Weyl modules in all these papers. The authors only considered the local Weyl modules defined via generators and relations. 

Basically, many of the results in characteristic zero remain valid in positive characteristic, but, unfortunately, it is unknown the adaptation of all these constructions from the aforementioned papers (except those already in the hyperalgebra context) to the positive characteristic setting, for instance a few results in \cite{CP,FL,CL,FoL} and \cite[Sections 5 and 6]{CFK}. This difficulty already appeared with many other results first proved in characteristic zero and then generalized to positive characteristic setting by using very different tools as in the pairs of papers \cite{CP,JM}, \cite{CFS,BM}, and \cite{FoL,BMM}.

The results of this work show that the Weyl functor and the Weyl modules defined in this paper satisfy properties similar to the ones satisfied by Weyl functors defined in the characteristic zero setting. 

Section \ref{preliminaries} is mostly dedicated to reviewing all preliminaries on Lie algebras, details of the construction of hyperalgebras, some useful straightening identities, as well as fixing the basic notation of the paper. Section \ref{categories} is dedicated to reviewing the relevant facts about the finite-dimensional modules for hyperalgebras $U_\mathbb F(\mathfrak g)$ and to define the main categories of objects for this paper. We define the Weyl functor and global and local Weyl modules in this section.  Our main results on the structure of these modules are in Subsection \ref{fd}.

\tableofcontents

\section{Preliminaries} \label{preliminaries}

Throughout this work, $\mathbb{C}$ denotes the set of complex numbers and $\mathbb{Z}$, $\mathbb{Z}_+$ and $\mathbb N$ are the sets of integers, non-negative integers, and positive integers respectively.

\subsection{Simple Lie algebras} 

Let $\mathfrak g$ be a finite-dimensional complex simple Lie algebra and $I$ the set of vertices of the associated Dynkin diagram. Fix a Cartan subalgebra $\mathfrak h$ of $\mathfrak g$ and let $R$ denote the corresponding set of roots. Let $\{\alpha_i:{i\in I}\}$ (respectively, $\{ \omega_i :{i\in I}\}$) denote the simple roots (respectively, fundamental weights). Set $Q=\bigoplus_{i\in I} \mathbb Z \alpha_i$, $Q^+=\bigoplus_{i\in I} \mathbb Z_+ \alpha_i$, $P=\bigoplus_{i\in I} \mathbb Z \omega_i$, $P^+=\bigoplus_{i\in I} \mathbb Z_+ \omega_i$, and $R^+=R\cap Q^+$. We denote the Weyl group of $\mathfrak g$ by $\mathcal W$ and its longest element by $w_0$.

Let $\mathcal{C}:=\{x_\alpha^{\pm},h_i:\alpha\in R^+,\ i\in I\}$ be a Chevalley basis of $\mathfrak g$ and set
$$x_i^\pm:=x_{\alpha_i}^{\pm}, h_\alpha:=[x_\alpha^+,x_\alpha^-], \text{ and } h_i:=h_{\alpha_i}.$$ For each $\alpha\in R^+$, the subalgebra of $\mathfrak g$ spanned by $\{x_{\alpha}^\pm,h_\alpha\}$ is naturally isomorphic to $\mathfrak{sl}_2$. We have a triangular decomposition $\mathfrak g=\mathfrak n^-\oplus \mathfrak h\oplus \mathfrak n^+$ with 
$$\mathfrak n^\pm:=\bigoplus_{\alpha\in R^+}\mathbb{C} x_\alpha^\pm.$$

\subsection{Map algebras}
Fix $\mathcal A$ a commutative $\mathbb C$-associative algebra with unity over $\mathbb C$ and suppose that $\mathcal A$ has a multiplicatively closed basis $\mathcal B$ (see \cite{BGRS} for details).

The map algebra of a Lie algebra $\mathfrak a$ over $\mathbb C$ is the $\mathbb C$-vector space $\mathfrak a_\mathcal A := \mathfrak a\otimes  \mathcal A$, with Lie bracket given by linearly extending the bracket
$$[g\otimes a, g'\otimes b]=[g,g']_{\mathfrak a}\otimes ab,\ g,g'\in\mathfrak a,\ a,b\in \mathcal A,$$
where $[\cdot,\cdot]_{\mathfrak a}$ is the Lie bracket of $\mathfrak a$.

The Lie algebra $\mathfrak a$ can be embedded into $\mathfrak a_\mathcal A $ as $\mathfrak a\otimes 1$ and, if $\mathfrak s$ is a subalgebra of $\mathfrak a$, then $\mathfrak s \otimes \mathcal A =:\mathfrak s_\mathcal A$ is naturally a subalgebra of $\mathfrak a_\mathcal A$. In particular, we have a decomposition
$$\mathfrak g_\mathcal A = \mathfrak n^+_\mathcal A\oplus \mathfrak h_\mathcal A \oplus \mathfrak n^-_\mathcal A$$
where $\mathfrak h_\mathcal A$ is an abelian subalgebra of $\mathfrak g_\mathcal A$.

\begin{rem} In the case that $\mathcal A$ is the coordinate ring of some algebraic variety $X$, it follows that $\mathfrak{a}_\mathcal A$ is isomorphic to the Lie algebra of regular maps $X\to\mathfrak a$.
\end{rem}

\subsection{Universal enveloping algebras}

For a Lie algebra $\mathfrak a$, the corresponding universal enveloping algebra of $\mathfrak a$ will be denoted by $U(\mathfrak a)$. The Poincar\'e--Birkhoff--Witt (PBW) Theorem implies that 
\begin{eqnarray*}
     U(\mathfrak g)&\cong& U(\mathfrak n^-)\otimes U(\mathfrak h)\otimes U(\mathfrak n^+)\\
     U(\mathfrak g_\mathcal A)&\cong& U(\mathfrak n^-_\mathcal A)\otimes U(\mathfrak h_\mathcal A)\otimes U(\mathfrak n^+_\mathcal A).
\end{eqnarray*}

The assignments $\Delta:\mathfrak g_\mathcal A\to U(\mathfrak g_\mathcal A)\otimes U(\mathfrak g_\mathcal A)$ where $x\mapsto x\otimes 1+1\otimes x$, $S:\mathfrak g_\mathcal A\to \mathfrak g_\mathcal A$ where $x\mapsto-x$, and $\epsilon:\mathfrak g_\mathcal A\to\mathbb C$ where $x\mapsto0$, can be  uniquely extended so that $U(\mathfrak g_\mathcal A)$ becomes a Hopf algebra with comultiplication $\Delta$, antipode $S$, and counit $\epsilon$.

The augmentation ideal (i.e., the kernel of $\epsilon$) for a Hopf algebra $H$ is denoted $H^0$.

\subsection{Integral forms and hyperalgebras} \label{hyperalgebras}

Given $\alpha\in R$ and $a\in \mathcal A\setminus\{0\}$, consider the power series with coefficients in $U(\mathfrak h_\mathcal A)$ given by
\begin{equation*}  
	\Lambda_{\alpha,a}(u) = \sum_{r=0}^\infty \Lambda_{\alpha,a,r}u^r = \exp \left(-\sum_{s=1}^{\infty}\frac{h_{\alpha}\otimes a^{s}}{s} u^{s}\right).
\end{equation*}

Note that $\Lambda_{\alpha,f,r}$ is a polynomial in $\left(h_\alpha\otimes f^j\right)$ for $j\in\{1,\ldots,r\}$. For $i\in I$, we simply write $\Lambda_{i,a,r}$ in place of $\Lambda_{\alpha_i,a,r}$.

The following lemma shows that elements of the form $\Lambda_{\alpha,a^k,r}$, $k\ge 2$, are linear combinations of products of elements of the form $\Lambda_{\alpha,a,s}$.
\begin{lem}\label{a^kintermsofa}
    Let $a\in\mathcal B$, $\alpha\in R$, and $r,k\in\mathbb N$. Then
    \begin{equation*}
	    \Lambda_{\alpha,a^k,r}=k\Lambda_{\alpha,a,kr}+\sum_{(\mathbf s,\mathbf n)}m_{\mathbf s,\mathbf n}\Lambda_{\alpha,a,s_1}^{n_1}\dots\Lambda_{\alpha,a,s_l}^{n_l}
	\end{equation*}
    where $m_{\mathbf s,\mathbf n}\in\mathbb Z$ and the sum is over all $\mathbf s,\mathbf n\in\mathbb N^l$ for some $l\in\mathbb N$ such that $s_i\neq s_j$, $l\sum n_js_j>1$, and $\sum n_js_j=rk$.  
\end{lem}
\begin{proof} The lemma is proven for $A=\mathbb C [t,t^{-1}]$ and $a=t$ as part of \cite[Lemma 5.11]{G}. We extend this to the current setting by replacing $t$ with $a$.
\end{proof}

The pure tensors in $\mathcal C\otimes\mathcal B$ form a basis for $\mathfrak g_\mathcal A$, which we shall denote by $\mathbb B$. Given an order on $\mathbb B$ and a PBW monomial according to this order, we construct an ordered monomial in the elements of
$$\mathcal M(\mathcal A)=\left\{(x^\pm_{\alpha}\otimes b)^{(k)}, \ \Lambda_{i,c,r}, \ \binom{h_{i}\otimes1}{k}\  \bigg|\ \alpha \in R^+,\ i\in I,\ b,c \in \mathcal B,\ c\ne 1,\ k,r\in \mathbb N\right\},$$ 
where 
$$(x^\pm_{\alpha}\otimes b)^{(k)}=\frac{(x^\pm_{\alpha}\otimes b)^k}{k!} \qquad \text{and} \qquad \binom{h_i}{k}=\frac{h_i(h_i-1)\dots(h_i-k+1)}{k!},$$
through the correspondence 
$$(x^\pm_{\alpha}\otimes b)^k \leftrightarrow (x^\pm_{\alpha}\otimes b)^{(k)}, \quad (h_{i}\otimes1)^k \leftrightarrow \binom{h_{i}\otimes1}{k}, \quad \text{ and } \quad (h_{i}\otimes c)^{r} \leftrightarrow\Lambda_{i,c,r}.$$

By using a similar correspondence we consider monomials in $U(\mathfrak g)$ formed by elements of
$$\mathcal M=\left\{ (x_{\alpha}^\pm)^{(k)},\tbinom{h_i}{k} \mid  \alpha\in R^+,\ i\in I, \ k\in\mathbb N\right\}.$$
It is clear that we have a natural inclusion $\mathcal M\subset \mathcal M(\mathcal A)$ and the set of ordered monomials thus constructed are bases of $U(\mathfrak g)$ and $U(\mathfrak g_\mathcal A)$, respectively.

Let  $ U_{\mathbb Z}(\mathfrak g)\subseteq  U(\mathfrak g)$ and $ U_{\mathbb Z}(\mathfrak g_\mathcal A)\subseteq  U(\mathfrak g_\mathcal A)$ be the $\mathbb Z$--subalgebras generated respectively by
$$\{(x_{\alpha}^\pm)^{(k)}\mid\alpha\in R^+,k\in\mathbb N\}\ \text{ and }\ \{(x^\pm_{\alpha}\otimes b)^{(k)}\mid\alpha\in R^+,b\in\mathcal B, k\in\mathbb N\}.$$

We have the following fundamental theorem which was proved in \cite{K2} for $ U_{\mathbb Z}(\mathfrak g)$ and in \cite{BC} for $ U_{\mathbb Z}(\mathfrak g_\mathcal A)$ (see also \cite{G,M} for other specific cases of these theorems).

\begin{thm}\label{gAforms}
    The subalgebra $U _{\mathbb Z}(\mathfrak g)$ (resp., $U _{\mathbb Z}(\mathfrak g_\mathcal A)$) is a free $\mathbb Z$-module and the set of all the ordered monomials constructed from $\mathcal M$ (resp., $\mathcal M(\mathcal A)$) is a $\mathbb Z$-basis of $U _{\mathbb Z}(\mathfrak g)$ (resp., $U _{\mathbb Z}(\mathfrak g_\mathcal A)$).  \hfill \qedsymbol
\end{thm}

Particularly, if $\mathfrak a\in\{\mathfrak g,\mathfrak n^\pm,\mathfrak h,\mathfrak g_\mathcal A,\mathfrak n^\pm_\mathcal A,\mathfrak h_\mathcal A\}$ and we set
$$ U_\mathbb Z(\mathfrak a):= U(\mathfrak a)\cap U_\mathbb Z(\mathfrak g_\mathcal A)$$
we have
$$\mathbb C\otimes_\mathbb Z U_\mathbb Z(\mathfrak a)\cong U(\mathfrak a).$$
Therefore, $ U_\mathbb Z(\mathfrak g_\mathcal A)$ and $ U _\mathbb Z(\mathfrak g)$ are \textit{integral forms} of $ U(\mathfrak g_\mathcal A)$ and $ U(\mathfrak g)$, respectively.

\begin{rem} $U _\mathbb Z(\mathfrak a)$ is a free $\mathbb Z$-module spanned by monomials formed by elements of $M\cap U(\mathfrak a)$ for the appropriate $M\in\{\mathcal M,\mathcal M(A)\}$. Notice that $U_\mathbb Z(\mathfrak g)= U(\mathfrak g)\cap U_\mathbb Z(\mathfrak g_\mathcal A)$ which allows us to regard $U_\mathbb Z(\mathfrak g)$ as a $\mathbb Z$-subalgebra of $U_\mathbb Z(\mathfrak g_\mathcal A)$.
\end{rem}

Given a field $\mathbb F$, the $\mathbb F$--hyperalgebra of $\mathfrak a$ is defined by
$$ U_\mathbb F(\mathfrak a):=\mathbb F\otimes_{\mathbb Z} U_\mathbb Z(\mathfrak a).$$

\

\noindent
\textbf{Notation:} we will keep denoting by $x$ the image of the element $x\in U_\mathbb Z(\mathfrak a)$ in $U_\mathbb F(\mathfrak a)$.

\

The PBW Theorem gives the following isomorphisms
\begin{eqnarray}
     U_\mathbb F(\mathfrak g)&\cong& U_\mathbb F(\mathfrak n^-) U_\mathbb F(\mathfrak h) U_\mathbb F(\mathfrak n^+)\\\label{tridecomp}
     U_\mathbb F(\mathfrak g_\mathcal A)&\cong& U_\mathbb F(\mathfrak n^-_\mathcal A) U_\mathbb F(\mathfrak h_\mathcal A) U_\mathbb F(\mathfrak n^+_\mathcal A)
\end{eqnarray}
and the Hopf algebra structure on $U(\mathfrak g_\mathcal A)$ induces a natural Hopf algebra structure over $\mathbb Z$ on $ U_\mathbb Z(\mathfrak g_\mathcal A)$ and this in turn induces a Hopf algebra structure on $ U_\mathbb F(\mathfrak g_\mathcal A)$.

\

We refer to the $\mathbb F$--hyperalgebra $U_\mathbb F(\mathfrak g_\mathcal A)$ as a \textit{hyper-map algebra} of $\mathfrak g$ and $\mathcal A$ over $\mathbb F$.

\begin{rem}\label{rem1} Recall that, if the characteristic of $\mathbb F$ is zero, the algebra $U_\mathbb F(\mathfrak g_\mathcal A)$ is isomorphic to $ U((\mathfrak g_\mathcal A)_\mathbb F)$, where $(\mathfrak g_\mathcal A)_\mathbb F:=\mathbb F\otimes_\mathbb Z(\mathfrak g_\mathcal A)_\mathbb Z$ and $(\mathfrak g_\mathcal A)_\mathbb Z$ is the $\mathbb Z$-span of $\mathbb{B}$. But, if $\mathbb F$ has positive characteristic, we have an algebra homomorphism $ U((\mathfrak g_\mathcal A)_\mathbb F)\to  U _\mathbb F(\mathfrak g_\mathcal A)$ that fails to be injective and to be surjective.
\end{rem}

\subsection{Straightening identities}
The next lemmas are essential tools in the proofs of Theorem \ref{gAforms} and it is also crucial in the study of finite-dimensional representations of hyper-map algebras. 
    
\begin{defn}
	Given $a,b\in A$ and $\alpha\in R^+$, we define the following series with coefficients in $U (\mathfrak{g}_\mathcal A)$:
	\begin{eqnarray*}
	  	X_{\alpha,a,b}^-(u)&=&\sum_{j=0}^\infty\left(x_\alpha^-\otimes a^{j}b^{j+1}\right)u^{j+1}.
	\end{eqnarray*}
Then, we set $\left(X^-_{\alpha,a,b}(u)\right)_n\in U(\mathfrak{g}_\mathcal A)$ to be the coefficient of $u^{n}$ in $X^-_{\alpha,a,b}(u)$. 
\end{defn}

The next lemma was originally proved in the $U(\mathfrak g\otimes \mathbb C[t,t^{-1}])$ setting in \cite[Lemma 7.5]{G} and in \cite[Lemma 5.4]{C} for a general map algebra $U(\mathfrak g_\mathcal A)$.  The first presentation in the context of hyperalgebras was in \cite[Lemma 1.4]{JM}.

\begin{lem}\label{basicrel} 
    Let $\alpha \in R^+$, $a,b\in\mathcal A$ and $r,s \in \mathbb N$ with $s\ge r\ge 1$. Then,
    $$(x^+_{\alpha}\otimes a)^{(r)}(x^-_{\alpha}\otimes b)^{(s)}\equiv(-1)^r\sum_{j=0}^r\left(\left(X^-_{\alpha,a,b}(u)\right)^{(s-r)}\right)_{s-r+j}\Lambda_{\alpha,ab,r-j} \quad \mod U_\mathbb F(\mathfrak g_\mathcal A)U(\mathfrak n^+_\mathcal A).$$ \hfill\qedsymbol
\end{lem}

Given $\alpha \in R^+$, $a\in \mathcal A$ and $k\ge 0$, define the degree of $(x^\pm_{\alpha} \otimes a)^{(k)}$ to be $k$. For a monomial of the form $(x^\pm_{\alpha_1}\otimes a_1)^{(k_1)}\cdots (x^\pm_{\alpha_l}\otimes a_l)^{(k_l)}$, where $\alpha_1,\ldots,\alpha_l\in R^+$, $a_1,\ldots,a_l\in A$, $k_1,\ldots,k_l\in\mathbb{Z}_+$, and the choice of $\pm$ fixed, define its degree to be $k_1+\cdots+k_l$.

\begin{lem}\label{commutrels} Let $\alpha,\beta\in R^+$,  $i\in I$, $a,b\in \mathcal A$, $k,l \in \mathbb N$.
    \begin{enumerate}
        \item \label{commutrels3} $\left(x^\pm_{\alpha}\otimes a\right)^{(k)}\left(x^\pm_{\beta} \otimes b\right)^{(l)}$ is in the $\mathbb Z$--span of elements $\left(x^\pm_{\beta} \otimes b\right)^{(l)}\left(x^\pm_{\alpha} \otimes a\right)^{(k)}$ and other monomials of degree strictly smaller than $k+l$.

        \item \label{commutrels1} $$\left(x_{\alpha}^+\right)^{(l)}\left(x_{\alpha}^-\right)^{(k)} = \sum_{m=0}^{\rm{min}\{k,l\}}\left(x_{\alpha}^-\right)^{(k-m)}\binom{h_{\alpha}-k-l+2m}{m}\left(x_{\alpha}^+\right)^{(l-m)}.$$

        \item \label{commutrels2} $$\binom{h_{i}}{l}\left(x^\pm_{\alpha}\otimes a\right)^{(k)} = \left(x^\pm_{\alpha} \otimes a\right)^{(k)}\binom{h_{i}\pm k\alpha(h_{i})}{l}.$$

        \item \label{commutrels4} $$\left(x^\pm_{\alpha}\otimes a\right)^{(k)}\left(x^\pm_{\alpha} \otimes a\right)^{(l)} = \binom{k+l}{k}\left(x^\pm_{\alpha} \otimes a\right)^{(k+l)}.$$
        
        \item \label{lambdax}
$$\Lambda_{\alpha,a,r}(x_\alpha^-\otimes b)^{(k)} =\sum_{s=0}^{r} \left(\sum_{j\ge 0}(j+1)(x_\alpha^-\otimes a^jb) \ u^{j} \right)_{r-s}^{(k)}\Lambda_{\alpha,a,s}.$$
    \end{enumerate}
\end{lem}

\begin{proof}
    The item (1) can be deduced from \cite[Equation (4.1.6)]{BC} and the item  (2) can be deduced from \cite[Lemma 26.2]{H}. (3) is proved by induction on $k+l$ and (4) is easily established. (5) is a reformulation of Proposition 4.1.2 (4.1.5) in \cite{BC} (also, see \cite[(8.12)]{G} and \cite[Lemma 4.3.4 (iii)]{M} for these formulas in the context of affine Kac-Moody algebras).
\end{proof}

\section{The categories of modules} \label{categories}

Let $\mathbb F$ be an algebraically closed field. The symbol $\otimes$ denotes the tensor product of $\mathbb F$-vector spaces.

 The following subsections, except \ref{sec:rev.g} and \ref{end}, are the positive characteristic partial counterpart of \cite[Sections 3 and 4]{CFK}.

\subsection{Finite-dimensional modules for hyperalgebras \texorpdfstring{$ U_\mathbb{F}(\mathfrak g)$}{}} \label{sec:rev.g}

We now review the finite-dimensional representation theory of $ U _\mathbb F(\mathfrak g)$ and refer to \cite[Section 2]{JM} for a more detailed review. This is motivation for what we will develop in the next sections.

Let $V$ be a $ U_\mathbb F(\mathfrak g)$-module. A nonzero vector $v\in V$ is called a weight vector if there exists $\mu\in U_\mathbb F(\mathfrak h)^*$ such that $hv=\mu(h)v$ for all $h\in U_\mathbb F(\mathfrak h)$. 
The subspace consisting of all weight vectors of weight $\mu$ is called weight space of weight $\mu$, which we denote by $V_\mu$. 
 When $V_\mu\ne 0$, $\mu$ is called a weight of $V$ and $\mathrm{wt} (V) = \{\mu\in U_\mathbb F(\mathfrak h)^*:V_\mu\ne 0\}$ is called the set of weights of $V$.
 If $V= \bigoplus_{\mu\in U_\mathbb F(\mathfrak h)^*} V_\mu$, then $V$ is said to be a weight module.
 
 Moreover, when $v \in V$ is a weight vector and $(x_\alpha^+)^{(k)} v = 0$ for all $\alpha\in R^+, k>0$, then $v$ is called a a highest-weight vector. When $V$ is generated by a highest-weight vector, it is said to be a highest-weight module. 
 
 Notice that we have an inclusion
 $P\hookrightarrow U_\mathbb F(\mathfrak h)^*$ determined by associating to $\mu\in P$ the functional (which we keep denoting $\mu$) given by
\begin{equation*}
\mu\left(\binom{h_i}{k}\right) := \binom{\mu(h_i)}{k} \quad\text{and}\quad \mu(xy):=\mu(x)\mu(y) \quad\text{for all}\quad  i\in I,k\ge 0, x,y\in U_\mathbb F(\mathfrak h).
\end{equation*}
In particular, this inclusion  
provides a partial order $\le$ on $U _\mathbb F(\mathfrak h)^*$ defined by $\mu\le\lambda$ if $\lambda-\mu\in Q^+$ and we have
\begin{equation}\label{e:xactonws}
    (x_\alpha^\pm)^{(k)} V_\mu \subseteq V_{\mu\pm k\alpha}\quad\text{for all}\quad \alpha\in R^+, \mu\in U_\mathbb F(\mathfrak h)^*, k>0.
\end{equation}

\begin{thm}\label{t:rh}
    Let $V$ be a $ U_\mathbb F(\mathfrak g)$-module.
    \begin{enumerate}
        \item If $V$ is finite-dimensional, then $V$ is a weight-module, $\mathrm{wt}  (V) \subseteq P$, and $\dim V_\mu = \dim V_{\sigma\mu}$ for all $\sigma\in\mathcal W, \mu \in U_\mathbb F(\mathfrak h)^\ast$. 
        
        \item \label{t:rh.weights} If $V$ is a highest-weight module of highest weight $\lambda$, then $\dim(V_{\lambda})=1$ and $V_{\mu}\ne 0$ only if $w_0\lambda \leq \mu\le \lambda$. Moreover, $V$  has a unique maximal proper submodule and a unique irreducible quotient. In particular, $V$ is indecomposable.
        
        \item{\label{t:rh.c}} For each $\lambda\in P^+$, the $ U_\mathbb F(\mathfrak g)$-module $W_\mathbb F(\lambda)$ given by the quotient of $ U_\mathbb F(\mathfrak g)$ by the left ideal $I_\mathbb F(\lambda)$ generated by
        \begin{equation*}
             U_\mathbb F (\mathfrak n^+)^0, \quad h-\lambda(h) \quad \text{and} \quad (x_\alpha^-)^{(k)}, \quad\text{for all}\quad h\in U_\mathbb F(\mathfrak h), \ \alpha\in R^+, \ k>\lambda(h_\alpha),
        \end{equation*}
        is nonzero and finite-dimensional. Moreover, every finite-dimensional highest-weight module of highest weight $\lambda$ is a quotient of $W_\mathbb F(\lambda)$.
        
        \item If $V$ is finite-dimensional and also irreducible, then there exists a unique $\lambda\in P^+$ such that $V$ is isomorphic to the irreducible quotient $V_\mathbb F(\lambda)$ of $W_\mathbb F(\lambda)$. If the characteristic of $\mathbb F$ is zero, then $W_\mathbb F(\lambda)$ is irreducible.
        
        \item\label{t:chWg} For each $\lambda\in P^+$, $\mathrm{ch} (W_\mathbb F(\lambda))$ is given by the Weyl character formula. In particular, $\mu\in\mathrm{wt} (W_\mathbb F(\lambda))$ if, and only if, $\sigma\mu\le\lambda$ for all $\sigma\in\mathcal W$. \hfill \qed
    \end{enumerate}
\end{thm}

The module $W_\mathbb F(\lambda)$ defined in Theorem \ref{t:rh} \eqref{t:rh.c} is called the \textit{Weyl module of highest weight $\lambda$}.

\begin{rem} 
    The notions of lowest-weight vectors and modules are similar and are obtained by replacing $(x_\alpha^+)^{(k)}$ by $(x_\alpha^-)^{(k)}$.
\end{rem}

\subsection{The category \texorpdfstring{$\mathcal I_\mathbb F$}{}  of \texorpdfstring{$U_\mathbb F(\mathfrak g_{\mathcal A})$}{}-modules and a weight-bounded subcategory} 

Let $V$ be a $ U_\mathbb F(\mathfrak g)$-module.  It is said that $V$ is locally finite-dimensional if any element of $V$ lies in a finite-dimensional $ U_\mathbb F(\mathfrak g)$-submodule of $V$. In other words, $V$ is isomorphic to a direct sum of irreducible finite-dimensional $U_\mathbb F(\mathfrak g)$-modules.

Let $\mathcal I_\mathbb F$ be the full subcategory of the category of $U_\mathbb F(\mathfrak g_{\mathcal A})$-modules which are locally finite-dimensional $U_\mathbb F(\mathfrak g)$-modules. Clearly, $\mathcal I_\mathbb F$ is an abelian category that is closed under tensor products. We shall abuse notation and write $V\in \mathcal I_\mathbb F$ to mean $V$ is an object of $I_\mathbb F$.

In the rest of this paper we shall use the following elementary result without mention: 

\begin{lem}\label{l:elem1} Let $V\in\mathcal  I_\mathbb F$.
	\begin{enumerate}
		\item If $V_\lambda\ne 0$ and  $\mathrm{wt}  (V)\subset\lambda-Q^+$, then $\lambda\in P^+$ and $U_\mathbb F(\mathfrak n_{\mathcal A}^+)^0V_\lambda=(x_\alpha^-)^{(s)}V_\lambda=0$, for all $\alpha\in R^+$, $s>\lambda(h_i)$.
		If, additionally, $V=U_\mathbb F(\mathfrak g_{\mathcal A})V_\lambda$ and $\dim V_\lambda=1$, then $V$ has a unique irreducible  quotient.
		\item If $V=U_\mathbb F(\mathfrak g_{\mathcal A})V_\lambda$ and $U_\mathbb F(\mathfrak n_{\mathcal A}^+)^0V_\lambda=0$, then $\mathrm{wt} (V)\subset\lambda-Q^+$.
		\item If $V$ is irreducible and finite-dimensional, then there exists $\lambda\in\mathrm{wt} (V)$ such that $\dim V_\lambda=1$ and $\mathrm{wt} (V)\subset\lambda- Q^+$. \hfill\qedsymbol
	\end{enumerate}
\end{lem}

Given a  left $U_\mathbb F(\mathfrak g)$-module $V$, by regarding $U_\mathbb F(\mathfrak g_{\mathcal A})$ as a right $U_\mathbb F(\mathfrak g)$-module via right multiplication, set
$$P(V):=U_\mathbb F(\mathfrak g_{\mathcal A})\otimes_{U_\mathbb F(\mathfrak g)} V.$$
Then, $P(V)$ is a left $U_\mathbb F(\mathfrak g_{\mathcal A})$-module by left multiplication and we have an isomorphism of $\mathbb F$-vector spaces
\begin{equation}\label{isov}
    P(V)\cong U_\mathbb F(\mathfrak g_{\mathcal A_+})\otimes V.
\end{equation}
where $\mathcal A_+$ is a fixed
vector space complement to the subspace $\mathbb C$ of $\mathcal A$. 

\begin{prop}\label{p:elem1}
	Let $V$ be a locally finite-dimensional $U_\mathbb F(\mathfrak g)$-module. Then,
	\begin{enumerate}
		\item $P(V)\in\mathcal  I_\mathbb F$.
		\item If $V\in\mathcal  I_\mathbb F$, then the map $P(V)\to V$ given by $u\otimes v\to uv$ is a surjective morphism of objects in $\mathcal  I_\mathbb F$. 
		\item If $V$ is projective in the category of locally finite-dimensional $U_\mathbb F(\mathfrak g)$-modules, then $P(V)$ is projective in $\mathcal  I_\mathbb F$. 
		\item For any $\lambda\in P^+$, we have $P(W_\mathbb F(\lambda))$ generated as a $U_\mathbb F(\mathfrak g_{\mathcal A})$-module by the element $p_\lambda:=1\otimes v_\lambda$ with defining relations
		\begin{equation*}\label{defproj}
		    U_\mathbb F(\mathfrak n^+)^0 p_\lambda=0,\quad {h}{} p_\lambda= {\lambda(h)}{} p_\lambda,\quad (x^-_{\alpha})^{(s)}p_\lambda =0,
		\end{equation*} 
		for all $h\in U_\mathbb F(\mathfrak{h}),\ \alpha\in R^+,\ s,k\in\mathbb Z_+, s>\lambda(h_i)$.\hfill\qedsymbol
	\end{enumerate}
\end{prop}

The proof of the above proposition is analogous to its characteristic zero counterpart \cite[Proposition 3]{CFK}.

\begin{rem}
	In characteristic zero $P(V)$ is always projective since all locally finite-dimensional $\mathfrak g$-modules are completely reducible and, hence, projective in the category of locally finite-dimensional $\mathfrak g$-modules. \hfill$\diamond$
\end{rem}

Given $\nu\in P^+$ and $V\in\mathcal  I_\mathbb F$, let $V^\nu\in\mathcal  I_\mathbb F$ be the unique  maximal $U_\mathbb F(\mathfrak g_{\mathcal A})$-quotient of
$V$ satisfying
\begin{equation}\label{defglobweylext}
    \mathrm{wt} (V^\nu)\subset \nu-Q^+,
\end{equation} 
or, equivalently,
$$V^\nu= V/\sum_{\mu\nleq\nu}U_\mathbb F(\mathfrak g_{\mathcal A})V_\mu.$$

Notice that any morphism $\pi:V\to V'$ of objects in $\mathcal  I_\mathbb F$ induces a morphism $\pi^\nu: V^\nu\to (V')^\nu$. 

Let $\mathcal  I_\mathbb F^\nu$ be the full subcategory of objects $V\in\mathcal  I_\mathbb F$ such that $V=V^\nu$. It follows from the finite-dimensional representation theory of hyperalgebras that
\begin{equation}\label{finiteset} V\in\mathcal  I_\mathbb F^\nu\implies \#\mathrm{wt} (V)<\infty,\end{equation}
since the weights of $V$ are bounded above and it is a locally finite-dimensional module. This implies that $V$ is a direct sum of finite-dimensional simple $U_\mathbb F(\mathfrak g)$-modules.

\

The next result is immediate.

\begin{lem} Let $\nu\in P^+$.

\begin{enumerate}
    \item $\mathcal  I_\mathbb F^\nu$ is an abelian category, but not a tensor subcategory of $\mathcal  I_\mathbb F$.
    \item If $V$ is projective in $\mathcal  I_\mathbb F^\nu$, then $P(V)^\nu$ is projective in $\mathcal  I_\mathbb F^\nu$
\end{enumerate}
 \hfill\qedsymbol
\end{lem}

\begin{rem}
    In the characteristic zero setting, part (2) is true for any $V$, cf. \cite[Corollary 3.2]{CFK}.
\end{rem}

\subsection{The global Weyl modules} We are now able to define the main object of study for this paper. It is a natural extension of the definition in the characteristic zero setting to the hyperalgebras:

\begin{defn} 
    The \textit{global Weyl module} of weight $\lambda\in P^+$ for $U_\mathbb F(\mathfrak g_{\mathcal A})$ is defined as
	$$\pmb W_\mathbb F(\lambda):=P(W_\mathbb F(\lambda))^\lambda,$$
\end{defn}

Let $w_\lambda$ be the image of $p_\lambda$ in $\pmb W_\mathbb F(\lambda)$.
The following proposition is essentially an immediate consequence of Proposition \ref{p:elem1} and provides a definition of $\pmb W_\mathbb F(\lambda)$ via generators and relations.  

\begin{rem}
The original definition given in \cite{CP} in the characteristic zero context was via generator and relations.
\end{rem}

\begin{prop}\label{p:weyl} 
    For  $\lambda\in P^+$, the module $\pmb W_\mathbb F(\lambda)$ is generated by $w_\lambda$ with the defining relations: 
    \begin{equation}\label{weyldef}
	    U_\mathbb F(\mathfrak n_{\mathcal A}^+)^0 w_\lambda=0,\quad  {h} w_\lambda={\lambda(h)}w_\lambda,\quad (x^-_{\alpha})^{(s)}w_\lambda =0,\ h\in U_\mathbb F(\mathfrak{h}), \ \alpha\in R^+,\  s,k\in\mathbb Z_+,\ s>\lambda(h_i).
	\end{equation}
\end{prop}

\proof The relation $U_\mathbb F(\mathfrak n_{\mathcal A}^+)^0 w_\lambda=0$ follows directly from the fact that $\pmb W_\mathbb F(\lambda) \subseteq \lambda-Q^+$. The other relations are valid since they are already satisfied by $p_\lambda$.  It remains to see that these are in fact all the relations we have for $\pmb W_\mathbb F(\lambda)$: let $W'$ be
the module generated by an element $w_\lambda$ with the giving relations. By Proposition \ref{p:elem1}, $W'$ is a quotient of $P(W_\mathbb F(\lambda))$. Further, $\mathrm{wt} (W')\subseteq \lambda-Q^+$ and it implies
that $W'$ satisfies \eqref{defglobweylext}. Finally, the maximality of $\pmb W_\mathbb F(\lambda)$ implies that $W'$ a quotient of $\pmb W_\mathbb F(\lambda)$.
\endproof

\subsection{The Weyl functor} 

Consider the annihilator of $w_\lambda$ in the $U_\mathbb F(\mathfrak g_{\mathcal A})$-module $\pmb W_\mathbb F(\lambda)$, i.e. 
\begin{equation*}
	\operatorname{Ann}_{U_\mathbb F(\mathfrak g_{\mathcal A})}(w_\lambda)=\left\{u\in U_\mathbb F(\mathfrak g_{\mathcal A}): uw_\lambda=0\right\} 
\end{equation*}
and set \begin{equation*}
  \operatorname{Ann}_{U_\mathbb F(\mathfrak h_{\mathcal A})}(w_\lambda) :=\operatorname{Ann}_{U_\mathbb F(\mathfrak g_{\mathcal A})}(w_\lambda)\cap U_\mathbb F(\mathfrak h_{\mathcal A}).
\end{equation*}
Clearly $\operatorname{Ann}_{U_\mathbb F(\mathfrak h_{\mathcal A})}(w_\lambda)$ is an ideal of $U_\mathbb F(\mathfrak h_{\mathcal A})$ and we set $${\mathcal A}_\mathbb F^\lambda:=\frac{U_\mathbb F(\mathfrak h_{\mathcal A})}{\operatorname{Ann}_{U_\mathbb F(\mathfrak h_{\mathcal A})}(w_\lambda)}.$$
	
Now we regard $\pmb W_\mathbb F(\lambda)$ as a right module for $U_\mathbb F(\mathfrak h_{\mathcal A})$ as follows:
$$(uw_\lambda)\cdot x:=ux\cdot w_\lambda, \; \forall  u\in U_\mathbb F(\mathfrak g_{\mathcal A}),\ x\in U_\mathbb F(\mathfrak h_{\mathcal A}).$$
To see that this action is well defined, one must prove that:
\begin{equation*}
	U_\mathbb F(\mathfrak n_{\mathcal A}^+)^0xw_\lambda=0,\ \ \left( {h} - \lambda(h)\right) xw_\lambda=0, \ \ (x_\alpha^-)^{(s)}xw_\lambda=0,
\end{equation*}
for all $h\in U_\mathbb F(\mathfrak{h})$, $\alpha\in R^+, x\in U_\mathbb F(\mathfrak h_{\mathcal A})$, and $s,k\in\mathbb Z_+, s>\lambda(h_i)$. However, the validity of the first two equalities are obvious and the third one follows from  $(x_\alpha^+)^{(s)}(xw_\lambda)=0$ for all $s>0$ and  $\pmb W_\mathbb F(\lambda)\in\mathcal  I_\mathbb F$. 
	
Moreover, for all $\mu\in P$, the subspaces $\pmb W_\mathbb F(\lambda)_\mu$ are $U_\mathbb F(\mathfrak h_{\mathcal A})$-modules for both the  left and right actions and
$$\operatorname{Ann}_{U_\mathbb F(\mathfrak h_{\mathcal A})}(w_\lambda)=\left\{u\in U_\mathbb F(\mathfrak h_{\mathcal A}): w_\lambda u=0=uw_\lambda\right\}=\left\{u\in U_\mathbb F(\mathfrak h_{\mathcal A}): \pmb W_\mathbb F(\lambda)u=0\right\}.$$
Then $\pmb W_\mathbb F(\lambda)$ is a $(U_\mathbb F(\mathfrak g_{\mathcal A}),{\mathcal A}_\mathbb F^\lambda)$-bimodule and each subspace $\pmb W_\mathbb F(\lambda)_\mu$ is a right
${\mathcal A}_\mathbb F^\lambda$-module. In particular,  $\pmb W_\mathbb F(\lambda)_\lambda$ is a ${\mathcal A}_\mathbb F^\lambda$-bimodule and we have an isomorphism of bimodules $$\pmb W_\mathbb F(\lambda)_\lambda\cong_{{\mathcal A}_\mathbb F^\lambda} {\mathcal A}_\mathbb F^\lambda.$$
	
\begin{defn} \label{loc}
    Let $\operatorname{mod} {\mathcal A}_\mathbb F^\lambda$ be the category of left ${\mathcal A}_\mathbb F^\lambda$-modules and 
	let 
	${\mathbf W}_\mathbb F^\lambda:\operatorname{mod}  {\mathcal A}_\mathbb F^\lambda\to\mathcal  I_\mathbb F^\lambda$ be the right exact functor given by
	$${\mathbf W}_\mathbb F^\lambda(M)=\pmb W_\mathbb F(\lambda)\otimes_{{\mathcal A}_\mathbb F^\lambda} M \text{ and } {\mathbf W}_\mathbb F^\lambda f=1\otimes f,$$
	where $M\in\operatorname{mod}  {\mathcal A}_\mathbb F^\lambda$ and $f\in\operatorname{Hom}_{{\mathcal A}_\mathbb F^\lambda}(M,M')$ for some $M'\in\operatorname{mod}  {\mathcal A}_\mathbb F^\lambda$. We call this functor the \textit{Weyl functor}.
\end{defn}
	
Once $\pmb W_\mathbb F(\lambda)\in \mathcal  I_\mathbb F$, the $U_\mathbb F(\mathfrak g)$-action on ${\mathbf W}_\mathbb F^\lambda (M)$ is also locally finite and so ${\mathbf W}_\mathbb F^\lambda (M)\in\mathcal  I_\mathbb F^\lambda$. The preceding discussion also shows that
$${\mathbf W}_\mathbb F^\lambda({\mathcal A}_\mathbb F^\lambda)\cong_{U_\mathbb F(\mathfrak g_{\mathcal A})} \pmb W_\mathbb F(\lambda) \text{ and }
{\mathbf W}_\mathbb F^\lambda (M)_\mu \cong_{{\mathcal A}_\mathbb F^\lambda}\pmb W_\mathbb F(\lambda)_\mu\otimes_{{\mathcal A}_\mathbb F^\lambda} M,$$
for all $\mu\in P,\ M\in\operatorname{mod}  {\mathcal A}_\mathbb F^\lambda$.

\begin{lem}\label{universal} 
	For all $\lambda\in P^+$ and $V\in\mathcal  I_\mathbb F^\lambda$ we have  $\operatorname{Ann}_{U_\mathbb F(\mathfrak h_{\mathcal A})}(w_\lambda)V_\lambda=0$.
\end{lem}
	
\proof 
    This is immediate from Lemma \ref{l:elem1} and Proposition \ref{p:weyl}.
\endproof
	
From this lemma, the left action of $U_\mathbb F(\mathfrak h_{\mathcal A})$ on $V_\lambda$ induces a left action of ${\mathcal A}_\mathbb F^\lambda$ on $V_\lambda$. We denote the resulting ${\mathcal A}_\mathbb F^\lambda$-module by ${\mathbf R}_\mathbb F^\lambda (V)$. Further, given $f \in\operatorname{Hom}_{\mathcal  I_\mathbb F^\lambda}(V,V')$ the restriction $f_\lambda:V_\lambda\to V'_\lambda$ is a morphism of ${\mathcal A}_\mathbb F^\lambda$-modules and the rules
$$V\to {\mathbf R}_\mathbb F^\lambda( V)\text{ and } f\to {\mathbf R}_\mathbb F^\lambda (f) := f_\lambda$$ 
define the functor 
$${\mathbf R}_\mathbb F^\lambda:\mathcal  I_\mathbb F^\lambda\to\operatorname{mod}  {\mathcal A}_\mathbb F^\lambda$$
that is exact, since restricting $f$ to a weight space is exact. 

\
	
The next theorem establishes that the Weyl functors satisfy
properties similar to the ones satisfied by the Weyl functors defined in the non-hyper setting. Particularly, item (4) gives a categorical definition of ${\mathbf W}_\mathbb F^\lambda (M)$. Its proof is identical to the proof in the characteristic zero setting, cf. \cite[\S3.7]{CFK} and \cite[\S4]{FMS}. 

\begin{thm}\label{defeta} Let $\lambda\in P^+$.
	\begin{enumerate}
	
	    \item  Given $M\in\operatorname{mod} {\mathcal A}_\mathbb F^\lambda$, as left ${\mathcal A}_\mathbb F^\lambda$-modules, we have  
${\mathbf R}_\mathbb F^\lambda{\mathbf W}_\mathbb F^\lambda (M) \cong \ M$. In particular, we have an isomorphism of functors $\operatorname{id}_{{\mathcal A}_\mathbb F^\lambda}\cong {\mathbf R}_\mathbb F^\lambda{\mathbf W}_\mathbb F^\lambda$.
	    \item Let $V\in\mathcal  I_\mathbb F^\lambda$. There exists a canonical map of $U_\mathbb F(\mathfrak g_{\mathcal A})$-modules $\eta_V: {\mathbf W}_\mathbb F^\lambda{\mathbf R}_\mathbb F^\lambda(V)\to V$ such that $\eta:{\mathbf W}_\mathbb F^\lambda{\mathbf R}_\mathbb F^\lambda\Rightarrow \operatorname{id}_{\mathcal I_\mathbb F^\lambda}$ is a natural transformation of functors and ${\mathbf R}_\mathbb F^\lambda$ is a right adjoint to ${\mathbf W}_\mathbb F^\lambda$.
	    \item The functor ${\mathbf W}_\mathbb F^\lambda$ maps projective objects to projective objects.
	    
	    \item    Let $V\in\mathcal  I_\mathbb F^\lambda$. \ $V\cong{\mathbf W}_\mathbb F^\lambda{\mathbf R}_\mathbb F^\lambda(V)$ if, and only if, for all $U\in\mathcal  I_\mathbb F^\lambda$ with $U_\lambda=0$, we have 
	\begin{equation}\label{altweyldef} 
		\operatorname{Hom}_{\mathcal  I_\mathbb F^\lambda}(V,U)=\operatorname{Ext}^1_{\mathcal  I_\mathbb F^\lambda}(V, U)=0.
	\end{equation}
	    \item The functor ${\mathbf W}_\mathbb F^\lambda$ is exact if, and only if, for all $U\in\mathcal  I_\mathbb F^\lambda$ with $U_\lambda=0$ we have 
	\begin{equation}\label{equivexact} 
		\operatorname{Ext}^2_{\mathcal  I_\mathbb F^\lambda}({\mathbf W}_\mathbb F^\lambda(M), U)=0\ \ \forall \ M\in\operatorname{mod}  {\mathcal A}_\mathbb F^\lambda. 
	\end{equation}
	\end{enumerate}
	\hfill \qedsymbol
\end{thm}

\subsection{Finite-dimensional and finitely generated modules} \label{fd}  The first part of next theorem was proved in \cite{CP} in the characteristic zero setting. The present context uses a similar approach to that in \cite{BiC,BM,JM}. 

\begin{thm} \label{fingen} 
    Let $\lambda\in P^+$ and assume $\mathcal A$ is finitely generated.
	\begin{enumerate}
	    \item ${\mathcal A}_\mathbb F^\lambda$ is a finitely generated $\mathbb F$-algebra.
		\item $\pmb W_\mathbb F(\lambda)$ is a finitely generated right ${\mathcal A}_\mathbb F^\lambda$-module. 
		\item If  $M\in\operatorname{mod}  {\mathcal A}_\mathbb F^\lambda$ is a  finitely generated (resp. finite-dimensional) then ${\mathbf W}_\mathbb F^\lambda(M)$ is a finitely generated (resp. finite-dimensional) $U_\mathbb F(\mathfrak g_{\mathcal A})$-module.
	\end{enumerate}
\end{thm}
	
	\proof For part (1), let $w_\lambda$ a generator for $\pmb W_\mathbb F(\lambda)$. Recall that $U_\mathbb F(\mathfrak h_{\mathcal A})$ is commutative and its elements are polynomials in $\Lambda_{i,a,r}$ and $\binom{h_i}{k}$ where $i\in I$, $r\in \mathbb Z_+$, $a\in A$, and $k\in \mathbb Z_+$. Now, since $\binom{h_i}{k} w_\lambda = \binom{\lambda(h_i)}{k} w_\lambda$ for all $i\in I$ and $k\in \mathbb Z_+$ we see that $\binom{h_i}{k} w_\lambda$ is zero  when $k>\lambda(h_i)$. Thus, to prove that ${\mathcal A}_\mathbb F^\lambda$ is a finitely generated $\mathbb F$-algebra, it suffices to prove that $\Lambda_{i,a,r} w_\lambda = 0$ for all but finitely many $r\in\mathbb Z_+$. This is immediate, since we have $\Lambda_{i,a,r} w_\lambda = 0$ for $r > \lambda(h_{i})$ by Proposition \ref{p:weyl} and Lemma \ref{basicrel}.

	For part (2), let $\{a_1,\dots,a_m\}$ be a generating set for $\mathcal A$.
	Given $\mathbf s=(s_1,\dots,s_m)\in\mathbb Z_+^{m}$, define $\mathbf a^{\mathbf s}:=a_1^{s_1}\dots a_m^{s_m}$ and $x^-_{\alpha,\mathbf s }:=x^-_{\alpha}\otimes \mathbf a^{\mathbf s}$.

    Using the decomposition \eqref{tridecomp}, we conclude that the elements of the set
    $$S=\left\{\left(x^-_{\beta_1,\mathbf{b}_1}\right)^{(n_1)}\cdots\left(x^-_{\beta_{\ell},\mathbf{b}_{\ell}}\right)^{(n_{\ell})}w_\lambda \mid  \beta_i \in R^+,\  \mathbf b_i\in\mathbb Z_+^m,\ \ell,n_i\in\mathbb N,\ i\in\{1,\ldots,\ell\}\right\}$$
    generate $\pmb W_\mathbb F(\lambda)$ as a right $\mathcal A_\mathbb F^\lambda$-module.
    
    In order to prove the statement, it suffices to prove that $\pmb W_\mathbb F(\lambda)$ is spanned by the elements
    $$\left(x_{\beta_1,{\mathbf s_1}}^-\right)^{(k_1)}\cdots\left(x_{\beta_r,{\mathbf s_r}}^-\right)^{(k_r)}w_\lambda,$$
    with $\mathbf s_1\ldots,\mathbf s_r\in \mathbb Z_+^m$ such that $\max(\mathbf s_j)<\lambda(h_{\beta_j})$ for all $j\in\{1,\ldots,r\}$, $\beta_1,\dots,\beta_r\in R^+$  and $\sum_{j}k_j\beta_j\le \lambda -w_0\lambda$. This last condition comes from the definition of the module $\pmb W_\mathbb F(\lambda)$, (cf. Sections 3.0 - 3.2). 
    
    Let $\mathcal R_\lambda = R^+\times \mathbb Z_+^m\times\mathbb Z_+$ and consider $\Xi$ as the set of functions $\xi:\mathbb N\to \mathcal R_\lambda$ such that $j\mapsto \xi_j=(\beta_j,\mathbf s_j,k_j)$ such that $k_j=0$ for all $j$ sufficiently large. Let $\Xi'$ be the subset of $\Xi$ consisting of the elements $\xi$ such that $\max(\mathbf s_j)<\lambda(h_{\beta_j})$ for all $j$.  
	
    Given $\xi\in\Xi$ we associate an element $v_\xi\in\pmb W_\mathbb F(\lambda)$ as follows $$v_\xi:=\left(x_{\beta_1,\mathbf s_1}^-\right)^{(k_1)}\cdots\left(x_{\beta_\ell, \mathbf s_\ell}^-\right)^{(k_\ell)}w_\lambda.$$ 
    Let $\mathcal S$ be the $\mathbb Z$-span of all vectors associated to elements belonging to $\Xi'$ and by $\mathcal S\cdot \mathcal A_\mathbb F^\lambda$ the set obtained by the right action of $\mathcal A_\mathbb F^\lambda$ on the elements of $\mathcal S$.  Define the degree of $\xi$ to be $d(\xi):= \sum_j k_j$ and the maximal exponent of $\xi$ to be $e(\xi):= \max\{k_j\}$. Notice that $e(\xi)\le d(\xi)$ and $d(\xi)\ne 0$ implies $e(\xi)\ne 0$. Since there is nothing to be proved when $d(\xi)=0$, we assume from now on that $d(\xi)>0$. 
    
    Let $\Xi_{d,e}$ be the subset of $\Xi$ consisting of those $\xi$ satisfying $d(\xi) = d$ and $e(\xi) = e$, and set $\Xi_d = \cup_{1\le e\le d} \Xi_{d,e}$.
    
    We prove by induction on $d$ and sub-induction on $e$ that if $\xi \in \Xi_{d,e}$ is such that there exists $j\in\mathbb{N}$ with $\max(\mathbf s_j) \ge \lambda(h_{\beta_{j}})$, then $v_\xi\in \mathcal S\cdot \mathcal A_\mathbb F^\lambda$.
    
    More precisely, given $0 < e \le d \in \mathbb N$, we assume, by the induction hypothesis, that this statement is true for every $\xi$ which belongs either to $\Xi_{d,e'}$ with $e'< e$ or to $\Xi_{d'}$ with $d'<d$.
    
    We split the proof according to two cases: $e=d$ and $e<d$. 
    
    \
    
\textbf{\textit{Step 1.}}    When $e=d$, it follows that $v_\xi=\left(x_{\alpha}^-\otimes a\right)^{(e)}w_\lambda$ for some $\alpha\in R^+$ and $a \in\mathcal A$. 
	    
    \begin{itemize}
	    \item[\textit{Step 1.1.}] If $e=1$, then $v_\xi=\left(x_{\alpha}^-\otimes a\right)w_\lambda$ for some $a \in\mathcal A$ and $\alpha\in R^+$. 
	    In order to prove this case, we first deal with the generators of $\mathcal A$ and then we deal with an arbitrary element in $\mathcal A$:

        By  Lemma \ref{basicrel}, for each $a_i\in \{ a_1,\dots,a_m\}$ and  $s\in\mathbb N$, with $s\geq\lambda(h_\alpha)$, we obtain:
        \begin{eqnarray}\label{S_1}
        0&=&(-1)^{s}\left(x_\alpha^+\otimes a_i\right)^{(s)} (x_\alpha^-\otimes1)^{(s+1)}w_\lambda\nonumber\\
        &=&\sum_{j=0}^{s}\left(x_\alpha^-\otimes a_i^j\right)\Lambda_{_\alpha,a_i,s-j}w_\lambda\nonumber\\
        &=&\sum_{j=0}^{s}\left(\left(x_\alpha^-\otimes a_i^j\right)w_\lambda\right)\Lambda_{_\alpha,a_i,s-j}.
        \end{eqnarray}
        So
        \begin{equation}\label{lem1}
        \left(x_\alpha^-\otimes a_i^s\right)w_\lambda=\sum_{j=0}^{s-1}\left(\left(x_\alpha^-\otimes a_i^j\right)w_\lambda\right) \left(-\Lambda_{_\alpha,a_i,s-j}\right).
        \end{equation}
        Thus, in particular, 
        $$\left(x_\alpha^-\otimes a_i^{\lambda(h_\alpha)}\right)w_\lambda\in\operatorname{span}\left\{\left(x_\alpha^-\otimes a_i^j\right)w_\lambda\mathcal A_\mathbb F^\lambda\ \big|\ 0\leq j<\lambda(h_\alpha)\right\} $$
        and an induction on $s$ gives 
        \begin{equation}\label{inS_1}
            \left(x_\alpha^-\otimes a_i^{s}\right)w_\lambda\in\operatorname{span}\left\{\left(x_\alpha^-\otimes a_i^j\right)w_\lambda\mathcal\mathcal A_\mathbb F^\lambda\ \big|\ 0\leq j<\lambda(h_\alpha)\right\},
        \end{equation}
        for each $i\in\{1,\dots,m\}$.
    
        Now, let $a_k\in\{ a_1,\dots,a_m\}$ such that $a_k\ne a_i$. 
       
        By applying $\left(h_\alpha\otimes a_k^{r}\right)$ to equation \eqref{S_1} we also see that $$\left(x_\alpha^-\otimes a_k^{r}a_i^{s}\right)w_\lambda\in \operatorname{span}\left\{\left(x_\alpha^-\otimes  a_i^j\right)w_\lambda\mathcal A_\mathbb F^\lambda,\left(x_\alpha^-\otimes  a_k^{r}a_i^j\right)w_\lambda\mathcal A_\mathbb F^\lambda\ \big|\ 0\leq j<\lambda(h_\alpha)\right\}.$$
        
        If $r,s<\lambda(h_\alpha)$, then $\left(x_\alpha^-\otimes a_i^{s}a_k^{r}\right)w_\lambda\in\mathcal S\cdot\mathcal A_\mathbb F^\lambda$. Otherwise, if $s<\lambda(h_\alpha)$ and $r\geq\lambda(h_\alpha)$, by applying $h_\alpha\otimes a_i^{s}$ to equation \eqref{S_1} (with $i=k$) we obtain:
        \begin{eqnarray*}
            0&=&(h\otimes a_i^{s})\sum_{j=0}^{r}\left(x_\alpha^-\otimes a_{k}^j\right)w_\lambda\Lambda_{_\alpha,a_{k},r-j}\\
            &=&\sum_{j=0}^{r}\left(x_\alpha^-\otimes a_{k}^j\right)w_\lambda(h\otimes a_i^{s})\Lambda_{_\alpha,a_{k},r-j}-2\sum_{j=0}^{r}\left(x_\alpha^-\otimes a_i^{s}a_k^j\right)w_\lambda\Lambda_{_\alpha,a_{k},r-j}.
        \end{eqnarray*}
        So, 
        \begin{eqnarray*}
            \left(x_\alpha^-\otimes a_i^{s}a_k^{r}\right)w_\lambda&=&\frac{1}{2}\sum_{j=0}^{r}\left(x_\alpha^-\otimes a_{k}^j\right)w_\lambda(h\otimes a_i^{s})\Lambda_{_\alpha,a_{k},r-j}\\
            &&-\sum_{j=0}^{r-1}\left(x_\alpha^-\otimes a_i^{s}a_k^j\right)w_\lambda\Lambda_{_\alpha,a_{k},r-j}.
        \end{eqnarray*}
        The first sum is in $\mathcal S \cdot \mathcal A_\mathbb F^\lambda$ by \eqref{inS_1} and the second is in $\mathcal S \cdot \mathcal A_\mathbb F^\lambda$ by a further induction on $r$. Thus, $\left(x_\alpha^-\otimes a_i^{s}a_k^{r}\right)w_\lambda\in\mathcal S \cdot \mathcal A_\mathbb F^\lambda$ for all $r,s\in\mathbb N$ and $i,k\in\{1,\dots,m\}$. 
        Hence, by repeating this argument for all generators of $\mathcal A$, we conclude that  $\left(x_\alpha^-\otimes a\right)w_\lambda\in \mathcal S \cdot \mathcal A_\mathbb F^\lambda$, for all $a\in\mathcal A$ and $\alpha\in R^+$.
        
         \item[\textit{Step 1.2}]
         
         If $e>1$, then $v_\xi=\left(x^-_{\alpha}\otimes \mathbf a^\mathbf r\right)^{(e)}w_\lambda$ for some $\alpha\in R^+$ and $\mathbf r=(r_1,\dots,r_m)\in\mathbb Z^m_+$. In this case, we first deal with the case $\mathbf r$ is a multiple of a canonical element $\mathbf e_i\in \mathbb Z_{+}^m$. So, suppose $\mathbf r=r_i\mathbf e_i$ for some $i\in\{1,\dots,m\}$, that means $$\left(x_\alpha^-\otimes \mathbf a^{\mathbf r}  \right)^{(e)}=\left(x_\alpha^-\otimes a_i^{r_i}  \right)^{(e)}.$$ 
         If $r_i< \lambda(h_{\alpha})$, there is nothing to prove. Otherwise, we first note that the defining relations of $\mathbf W_\mathbb F(\lambda)$ and Lemma \ref{basicrel} imply that

\begin{eqnarray*}
        0&=&(-1)^{er_i}\left(x_\alpha^+\otimes a_i\right)^{(er_i)}(x_\alpha^- \otimes1)^{(er_i+e)}w_\lambda\\
        &=&\sum_{j=0}^{er_i}\left(\left(X^-_{\alpha,a_i,1}(u)\right)^{(e)}\right)_{j+e}\Lambda_{\alpha,a_i,er_i-j}w_\lambda\\
        &=&\sum_{j=0}^{er_i}\left(\left(\sum_{z=0}^\infty\left(x_\alpha^-\otimes a_i^{z}\right)u^{z+1}\right)^{(e)}\right)_{j+e}\Lambda_{\alpha,a_i,er_i-j}w_\lambda\\
        &=&\sum_{j=0}^{er_i}\left(\left(\left(\sum_{z=0}^\infty\left(x_\alpha^-\otimes a_i^{z}\right)u^{z+1}\right)^{(e)}\right)_{j+e}w_\lambda\right)\Lambda_{\alpha,a_i,er_i-j}\\
         &=&\sum_{j=0}^{er_i}\left(\sum_{\substack{t\in\mathbb Z_+\\0\le t\leq r_i\\te=j}}\left(x^-_\alpha\otimes a_i^t\right)^{(e)}w_\lambda\right)\Lambda_{\alpha,a_i,er_i-j}+\mathcal Xw_\lambda\\
         &=&\left(x_\alpha^-\otimes a_i^{r_i}\right)^{(e)}w_\lambda+\sum_{j=0}^{er_i-1}\left(\sum_{\substack{t\in\mathbb Z_+\\0\le t<r_i\\te=j}}\left(x_\alpha^-\otimes a_i^t\right)^{(e)}w_\lambda\right)\Lambda_{\alpha,a_i,er_i-j}+\mathcal Xw_\lambda 
        \end{eqnarray*}
    where $\mathcal Xw_\lambda$ belongs to the $\mathcal A_\mathbb F^\lambda$-span of vectors $v_{\phi}$ with $\phi\in \Xi_{e,e'}$, for $e'<e$. Thus, 
        \begin{eqnarray*}\label{eq2}
        \left(x_\alpha^-\otimes a_i^{r_i}\right)^{(e)}w_\lambda&=&-\sum_{j=0}^{er_i-1}\left(\sum_{\substack{t\in\mathbb Z_+\\0\le t<r_i\\te=j}}\left(x_\alpha^-\otimes a_i^t\right)^{(e)}w_\lambda\right)\Lambda_{\alpha,a_i,er_i-j}+\mathcal Xw_\lambda,
        \end{eqnarray*}
    the term $\mathcal Xw_\lambda$ fits the induction hypothesis, and by a further induction on $r_i\ge \lambda(h_\alpha)$ we conclude that $\left(x_\alpha^-\otimes a_i^{r_i}\right)^{(e)}w_\lambda \in \mathcal S \cdot \mathcal A_\mathbb F^\lambda$.

        \item[\textit{Step 1.3}] Still in the case $e>1$, we know consider the arbitrary case of 
        $v_\xi=\left(x^-_{\alpha}\otimes \mathbf a^\mathbf r\right)^{(e)}w_\lambda$ for some $\alpha\in R^+$ and $\mathbf r=(r_1,\dots,r_m)\in\mathbb Z^m_+$, with $\mathbf r$ not a multiple of a canonical element in $\mathbb Z_{+}^{m}$. In this case, if $r_i< \lambda(h_{\alpha})$ for $i=1,\dots,m$, there is nothing to prove. Otherwise, we have $r_i\geq \lambda(h_{\alpha})$ for some $i\in\{1,\dots,m\}$. Set $\mathbf r'= \mathbf r-r_i\mathbf e_i$.

        From Step 1.2, we have 
        $$0=\left(x_\alpha^-\otimes a_i^{r_i}\right)^{(e)}w_\lambda+\sum_{j=0}^{er_i-1}\left(\sum_{\substack{t\in\mathbb Z_+\\0\le t<r_i\\te=j}}\left(x_\alpha^-\otimes a_i^{t}\right)^{(e)}w_\lambda\right)\Lambda_{\alpha,a_i,er_i-j}+\mathcal Xw_\lambda$$ 
        where $\mathcal Xw_\lambda$ belongs to the $\mathcal A_\mathbb F^\lambda$-span of vectors $v_{\phi}$ with $\phi\in \Xi_{e,e'}$, for $e'< e$. Thus, by applying $\Lambda_{\alpha,\mathbf a^{\mathbf r'},e}$ to this equation, from Lemma \ref{commutrels}\eqref{lambdax}, with $k=r=e, a=\mathbf a^{\mathbf r'}$ and $b=a_i^{r_i}$, we get
        \begin{eqnarray*}
        0&=&\Lambda_{\alpha,\mathbf a^{\mathbf r'},e}\left(x_\alpha^-\otimes a_i^{r_i}\right)^{(e)}w_\lambda +\Lambda_{\alpha,\mathbf a^{\mathbf r'},e}\sum_{j=0}^{er_i-1}\left(\sum_{\substack{t\in\mathbb Z_+\\te=j}}\left(x_\alpha^-\otimes a_i^{t}\right)^{(e)}w_\lambda\right)\Lambda_{\alpha,a_i,er_i-j}\\
        &&+\Lambda_{\alpha,\mathbf a^{\mathbf r'},e}\mathcal Xw_\lambda\\ 
        &=&2\left(x_\alpha^-\otimes \mathbf a^{\mathbf r}  \right)^{(e)}w_\lambda
        +\sum_{s=1}^{e-1}\left(\sum_{j\ge 0}(j+1)\left(x_\alpha^-\otimes {\mathbf a^{\mathbf r'}}^j a_i^{r_i}\right)u^{j}\right)_{e-s}^{(e)}w_\lambda\Lambda_{\alpha,\mathbf a^{\mathbf r'},s}
        \\
        &&+\left(x_\alpha^-\otimes a_i^{r_i}\right)^{(e)}w_\lambda\Lambda_{\alpha,\mathbf a^{\mathbf r'},e}+2\sum_{j=0}^{er_i-1}\sum_{\substack{t\in\mathbb Z_+\\te=j}} \left(x_\alpha^-\otimes \mathbf a^{\mathbf r'} a_i^{t}\right)^{(e)}w_\lambda\Lambda_{\alpha,a_i,er_i-j}\\
        &&+\sum_{j=0}^{er_i-1}\sum_{\substack{t\in\mathbb Z_+\\te=j}} \sum_{s=1}^{e-1}\left(\sum_{k\ge 0}(k+1)\left(x_\alpha^-\otimes {\mathbf a^{\mathbf r'}}^ka_i^{t}\right)u^{k}\right)_{e-s}^{(e)}w_\lambda\Lambda_{\alpha,a_i,er_i-j}\Lambda_{\alpha,\mathbf a^{\mathbf r'},s}\\
        &&+\sum_{j=0}^{er_i-1}\sum_{\substack{t\in\mathbb Z_+\\te=j}} \left(x_\alpha^-\otimes a_i^{t}\right)^{(e)}w_\lambda\Lambda_{\alpha,a_i,er_i-j}\Lambda_{\alpha,\mathbf a^{\mathbf r'},e}
        +\Lambda_{\alpha,\mathbf a^{\mathbf r'},e}\mathcal Xw_\lambda\\
        &=&2\left(x_\alpha^-\otimes\mathbf a^{\mathbf r}\right)^{(e)}w_\lambda
        +\mathcal X_1w_\lambda+2\sum_{j=0}^{er_i-1}\sum_{\substack{t\in\mathbb Z_+\\te=j}} \left(x_\alpha^-\otimes \mathbf a^{\mathbf r'} a_i^{t}\right)^{(e)}w_\lambda\Lambda_{\alpha,a_i,er_i-j}\\
        &&+\mathcal X_2w_\lambda+\sum_{j=0}^{er_i}\sum_{\substack{t\in\mathbb Z_+\\te=j}} \left(x_\alpha^-\otimes a_i^{t}\right)^{(e)}w_\lambda\Lambda_{\alpha,a_i,er_i-j}\Lambda_{\alpha,\mathbf a^{\mathbf r'},e}+\mathcal X_3w_\lambda
        \end{eqnarray*}
        where $\mathcal X_1w_\lambda$, $\mathcal X_2w_\lambda$, and $\mathcal X_3w_\lambda$, belong to the $\mathcal A_\mathbb F^\lambda$-span of vectors $v_{\phi}$ with $\phi\in \Xi_{e,e'}$, for $e'< e$. Thus, these terms fit the induction hypotheses and we can write
        \begin{eqnarray}\label{concl1case}
            \left(x_\alpha^-\otimes \mathbf a^{\mathbf r}\right)^{(e)}w_\lambda
            &=&-\sum_{j=0}^{er_i-1}\sum_{\substack{t\in\mathbb Z_+\\0\leq t<r_i\\te=j}} \left(x_\alpha^-\otimes \mathbf a^{\mathbf r'} a_i^{t}\right)^{(e)}w_\lambda\Lambda_{\alpha,a_i,er_i-j}'\nonumber\\
            &&-\frac{1}{2}\sum_{j=0}^{er_i}\sum_{\substack{t\in\mathbb Z_+\\te=j}}\left(x_\alpha^-\otimes a_i^{t}\right)^{(e)}w_\lambda\Lambda_{\alpha,a_i,er_i-j}\Lambda_{\alpha,\mathbf a^{\mathbf r'},e}+\mathcal X_4w_\lambda,
        \end{eqnarray}
 where $\mathcal X_4w_\lambda$ belongs to the $\mathcal A_\mathbb F^\lambda$-span of vectors $v_{\phi}$ with $\phi\in \Xi_{e,e'}$, for $e'< e$. The last summation on the right-hand side is in $\mathcal S\cdot \mathcal A_\mathbb F^\lambda$ by Step 1.2. Now we again argue by a further induction on $r_i\ge \lambda(h_\alpha)$ that $\left(x_\alpha^-\otimes\mathbf a^{\mathbf r}\right)^{(e)}w_\lambda$ is in the span of vectors from $\Xi_e$ where the power of the generator $a_i$ is smaller than $\lambda(h_\alpha)$. We now repeat the same argument for each one of the generators of $\mathcal A$ in the elements on the right hand side of \eqref{concl1case} and get $\left(x_\alpha^-\otimes\mathbf a^{\mathbf r}  \right)^{(e)}w_\lambda\in \mathcal S \cdot \mathcal A_\mathbb F^\lambda$.
 
 \item[\textit{Step 2}] The remaining case $e=e(\xi)<d(\xi)=d$ is handled with a simple repeated application of Lemma \ref{commutrels}\eqref{commutrels3} and the induction hypothesis completes the proof of part (2).
\end{itemize} 

\

 Finally, part (3) is immediate from part (2), Definition \ref{loc}, and \eqref{finiteset}. 
 \endproof

\subsection{Simple modules in \texorpdfstring{$\mathcal  I_\mathbb F^\lambda$}{} and local Weyl modules} Let $\rm{irr}(\operatorname{mod} {\mathcal A}_\mathbb F^\lambda$) be the set of irreducible representations of ${\mathcal A}_\mathbb F^\lambda$. Since ${\mathcal A}_\mathbb F^\lambda$ is a commutative finitely generated algebra it follows that if $M\in\rm{irr}(\operatorname{mod}  {\mathcal A}_\mathbb F^\lambda)$ then $\dim M=1$. By Theorem \ref{fingen} we see that 
$$\dim{\mathbf W}_\mathbb F^\lambda (M)<\infty\text{ and }{\mathbf R}_\mathbb F^{\lambda}{\mathbf W}_\mathbb F^\lambda (M) =M,\text{ for all }M\in\rm{irr}(\operatorname{mod}  {\mathcal A}_\mathbb F^\lambda).$$ 

\begin{defn} The $U_\mathbb F(\mathfrak g_\mathcal A)$-module $\mathbf W_\mathbb F^\lambda (M)$, where $M$ is an irreducible object of $\operatorname{mod} A_\mathbb F^\lambda$, is called \textit{local Weyl module}.
\end{defn}

Consider $\mathbf V_\mathbb F^\lambda(M)$ as the unique irreducible quotient of ${\mathbf W}_\mathbb F^\lambda(M)$ (see Lemma \ref{l:elem1}). The next result shows that any irreducible module in $\mathcal  I_\mathbb F^\lambda$ is isomorphic to $\mathbf V_\mathbb F^\mu(M)$ for some $\mu\in P^+$. 
	
\begin{prop}
\label{uniqueirr} 	Let $\lambda\in P^+$ and assume that $V\in\mathcal  I_\mathbb F^\lambda$ is irreducible.

\begin{enumerate}

\item There exists $\mu\in P^+\cap \mathcal (\lambda- Q^+)$ such that $\mathrm{wt}  V\subset\mu- Q^+ \text{ and } \dim V_\mu=1$. In particular,  $V$ is the unique irreducible  quotient of ${\mathbf W}_\mathbb F^\mu{\mathbf R}_\mathbb F^\mu(V)$ and $\dim V<\infty$. 
\item If $V'\in\mathcal  I_\mathbb F$ we have  $V\cong V'$ as $U_\mathbb F(\mathfrak g_{\mathcal A})$-modules if, and only if, ${\mathbf R}_\mathbb F^\mu(V)\cong {\mathbf R}_\mathbb F^{\mu}(V')$ as ${\mathcal A}_\mathbb F^\mu$-modules.
\end{enumerate}
\end{prop}
	
\proof 
    Since $V\in \mathcal  I_\mathbb F^\lambda$, it follows that there exists  $\mu\in\lambda-Q^+$ with  $ V_\mu\ne 0$ and $U_\mathbb F(\mathfrak n_{\mathcal A}^+)V_\mu=0$.
	It is immediate from Theorem \ref{defeta} that $V$ is a quotient of ${\mathbf W}_\mathbb F^\mu {\mathbf R}_\mathbb F^\mu(V)$. If  $S_\mu$ is a nonzero proper $U_\mathbb F(\mathfrak h_{\mathcal A})$-submodule of $V_\mu$, then $U_\mathbb F(\mathfrak g_{\mathcal A})S_\mu$ is a proper submodule of $V$, a contradiction. Hence, ${\mathbf R}_\mathbb F^\mu(V)$ is an irreducible  ${\mathcal A}_\mathbb F^\mu$-module which implies that $\dim V_\mu=1$. Theorem \ref{fingen} now implies that $\dim \mathbf W_\mathbb F^\mu \mathbf R_\mathbb F^\mu(V)<\infty$ and $\dim V<\infty$. The proof that $V$ is the unique irreducible quotient of $\mathbf W_\mathbb F^\mu \mathbf R_\mathbb F^\mu(V)$ is standard once we have $\mathbf R_\mathbb F^{\mu} \mathbf W_\mathbb F^{\mu} \mathbf R_\mathbb F^{\mu}(V)\cong V_{\mu}$.
	The final statement of the lemma is now deduced from the first part. 
\endproof

\subsection{On endomorphisms of Global Weyl modules} \label{end}

The global Weyl modules are the natural objects to play a role similar to that of the Verma modules  in the study of the representations of $U_\mathbb F(\mathfrak g\otimes \mathcal A)$. One of the fundamental properties of Verma modules is that the space of morphisms between two Verma modules is either zero or one–-dimensional. Searching for the analogue of this property for global Weyl modules, the situation is more complicated and partial results are obtained under certain restrictions on $\mathfrak g$, $\lambda$ and $\mathcal A$, even in characteristic zero, cf. \cite{BCGM}. We have this first similar (small) advance, whose proof follows that of  \cite[Lemma 1.11]{BCGM}.

\begin{prop} 
    For $\lambda\in P^+$, $\pmb a \in {\mathcal A}_\mathbb F^\lambda$ the assignment $w_\lambda\to w_\lambda\pmb a$ extends to a homomorphism $\pmb W_\mathbb F(\lambda)\to \pmb W_\mathbb F(\lambda)$ of $U_\mathbb F(\mathfrak g_{\mathcal A})$-modules and we have 
	$$\operatorname{Hom}_{U_\mathbb F(\mathfrak g_{\mathcal A})}(\pmb W_\mathbb F(\lambda),\pmb W_\mathbb F(\lambda))\cong_{{\mathcal A}_\mathbb F^\lambda}  {\mathcal A}_\mathbb F^\lambda. $$
	\end{prop}
\proof From the $(U_\mathbb F(\mathfrak g_{\mathcal A}),{\mathcal A}_\mathbb F^\lambda)$-bimodule structure of $\pmb W_\mathbb F(\lambda)$, we can see that $w_\lambda\pmb a$ satisfies the defining relations of $\pmb W_\mathbb F(\lambda)$ which yields the first statement of the proposition. For the second, let $\pi:\pmb W_\mathbb F(\lambda)\to \pmb W_\mathbb F(\lambda)$ be a nonzero $U_\mathbb F(\mathfrak g_{\mathcal A})$-module map. Since $\pmb W_\mathbb F(\lambda)_\lambda = U_\mathbb F(\mathfrak h_{\mathcal A})w_\lambda$, there exists $u_\pi\in U(\mathfrak h_{\mathcal A})$ such that $\pi(w_\lambda)= u_\pi w_\lambda$. Since $\pi$ is nonzero, the image $\tilde u_\pi$ of $u_\pi$ in ${\mathcal A}_\mathbb F^\lambda$ is nonzero. Thus, we obtain a well-defined map $\operatorname{Hom}_{U_\mathbb F(\mathfrak g_{\mathcal A})}(\pmb W_\mathbb F(\lambda),\pmb W_\mathbb F(\lambda))\to {\mathcal A}_\mathbb F^\lambda$ given by $\pi\mapsto \tilde u_\pi$, which is an isomorphism of right ${\mathcal A}_\mathbb F^\lambda$-modules.
\endproof

	\

\subsection*{Future projects} on this subject could include the decomposition of the global Weyl modules as $\mathfrak g$-modules, the freeness of global Weyl modules as $\mathcal A_\mathbb F^\lambda$-modules, the structure of the Weyl modules for fundamental weights, the independence of the choice of the irreducible module for $\mathcal A_\mathbb F^\lambda$ to construct the local Weyl modules, the structure of the algebras $\mathcal A_\mathbb F^\lambda$ and their irreducible representations, in a similar fashion to \cite[Sections 6 and 7]{CFK}, which was inspired by \cite{C1,L,R}.

\end{document}